\documentclass{article}
\usepackage[round]{natbib}
\usepackage[title]{appendix}
\usepackage{amssymb}
\usepackage{amsmath}
\usepackage{float, graphicx}
\usepackage{color}
\usepackage{amsthm}
\RequirePackage[colorlinks,citecolor=blue,urlcolor=blue,linkcolor=blue]{hyperref}

\newtheorem{theorem}{Theorem}
\newtheorem{lemma}[theorem]{Lemma}

\newtheoremstyle{named}{}{}{\itshape}{}{\bfseries}{.}{.5em}{\thmnote{#3 }#1}
\theoremstyle{named}
\newtheorem*{namedtheorem}{Theorem}
\newtheorem*{namedproposition}{Proposition}

\begin{document}

\title{Optimal Control of the SIR\ Model with Constrained Policy, with
an Application to COVID-19}
\author{Yujia Ding, Henry Schellhorn}
\maketitle

\begin{abstract}
This article considers the optimal control of the SIR\ model with both
transmission and treatment uncertainty. It follows the model presented in
\cite{gatto2021optimal}. We make four significant improvements on the
latter paper. First, we prove the existence of a solution to the model. Second,
our interpretation of the control is more realistic: while in Gatto and
Schellhorn the control $\alpha $ is the proportion of the population that
takes a basic dose of treatment, so that $\alpha >1$ occurs only if some
patients take more than a basic dose, in our paper, $\alpha $ is constrained
between zero and one, and represents thus the \textit{proportion of the
population} undergoing treatment. Third, we provide a complete solution for
the moderate infection regime (with constant treatment). Finally, we give a
thorough interpretation of the control in the moderate infection regime,
while Gatto and Schellhorn focussed on the interpretation of the low
infection regime. Finally, we compare the efficiency of our control to curb
the COVID-19 epidemic to other types of control.
\end{abstract}

\section{Introduction}

This article extends the analysis of the model presented in \cite{gatto2021optimal}.We make four significant improvements on the latter paper.
First, we prove existence of a solution. Second, In \cite{gatto2021optimal} the
optimal control $\alpha $ has the interpretation of the proportion of the
population that takes a basic dose of treatment, so that $\alpha >1$ occurs
only if a proportion of the population takes more than a basic dose of
treatment. In the low infection regime part of our paper, $\alpha $ is
constrained to be between zero and one, and represents thus the \textit{%
proportion of the population} undergoing treatment. The latter
interpretation is much more realistic, as it is uncommon to ration
treatment. Third, we provide a complete solution for the moderate infection
regime (with constant treatment). The final improvement is a thorough
numerical analysis and sensitivity analysis of the moderate infection
regime, while \cite{gatto2021optimal} focused exclusively on the
interpretation of the control in the low infection regime. This enables us
to discover some errors in the second-order term of the solution in \cite{gatto2021optimal}, which we correct here. Finally, we compare the efficiency of
our control to curb the COVID-19 pandemic to other types of control. While
our solution is complex, it allows to satisfy the objective better. Also,
our analytical solutions allow for an intuitive understanding of the
optimal control compared to a purely numerical solution.

The structure of the article is as follows. In section \ref{sec::intro} we briefly
introduce the model in \cite{gatto2021optimal}, and provide a proof of existence
of the solution. In section \ref{sec::low}, we show our results for the low infection
regime. In section \ref{sec::moderate}, we extend and analyze the solution in the moderate
infection regime. Section \ref{sec::application} shows our experimental results when applying our
methodology to the COVID-19 in the US in 2020. We draw the conclusion in Section \ref{sec::conclusion}. We refer the reader to our earlier paper, \cite{gatto2021optimal} for a literature review.

\section{A Stochastic SIR\ Model with Treatment Uncertainty}
\label{sec::intro}
Let $S$, $I$, $R$ be the proportions of susceptible, infected,
and out of infection (recovered, and dead), respectively. Let $\beta $ be the transmission
rate and $\mu $ be the death rate.

In the SIR model, the rate of decrease $\frac{dS}{dt}$ of the proportion of
susceptible is equal to the constant transmission rate $\beta $ time $SI$.
As in \cite{gatto2021optimal}, we add a term $\sigma _{S}\sqrt{SI}\frac{dB_{1}}{%
dt}$, where $\frac{dB_{1}}{dt}$ is white noise, in order to model the error
in the transmission rate:
\begin{equation*}
\frac{dS}{dt}=-\beta SI+\sigma _{S}\sqrt{SI}\frac{dB_{1}}{dt}  \label{S}
\end{equation*}

The optimal policy $\alpha $ is the proportion of the infected population
that received treatement, thus $\alpha (t)\in \lbrack 0,1]$. The presence of
this constraint is an important addition to the model in \cite{gatto2021optimal}. Depending whether the individual is treated or not, there are
then four different ways for an infected individual to exit the pool of
infected:
\begin{itemize}
\setlength\itemsep{-0.1em}
\item not treated and recover
\item not treated and die
\item treated and recover
\item treated and died.
\end{itemize}
Thus, the "out of infection rate" will be:
\begin{gather}
\label{dR/dt}
\begin{aligned}
\frac{dR(t)}{dt} =&\underset{\text{not treated and recover}}{\underbrace{%
(1-\alpha (t))I(t)K_{0}}}+\underset{\text{not treated and die}}{\underbrace{%
(1-\alpha (t))I(t)\mu _{0}}}+\underset{\text{treated and recover}}{%
\underbrace{\alpha (t)I(t)K_{1}(t)}}   \\
&+\underset{\text{treated and die}}{\underbrace{\alpha (t)I(t)\mu _{1}}}-%
\underset{\text{treatment measurement error}}{\underbrace{\alpha
(t)I(t)\sigma \frac{dB_{2}}{dt}}}
\end{aligned}
\end{gather}

For simplicity, we assume that the Brownian motion driving transmission
uncertainty ($B_{1}$) is independent from the Brownian motion driving
treatment uncertainty ($B_{2}$). We suppose that $\mu _{0}\geq \mu _{1}$
(people die faster without treatment than with treatment), but the reader
will not lose any intuition by supposing that $\mu _{0}=\mu _{1}$. Most of
the time $K_{1}(t)>K_{0}$ (treatment is better than no treatment), but not
necessarily. We relax this requirement somewhat by requiring:
\begin{equation}
P(K_{0}<K_{1}(t))\text{ is close to one}  \label{positivity}
\end{equation}

We model the treatment rate as an Ornstein-Uhlenbeck process:%
\begin{equation*}
dK_{1}(t)=\lambda _{k}(\bar{k}_{1}-K_{1}(t))dt+\sigma _{k}dB_{2}(t)
\end{equation*}%
with the mean-reversion rate $\lambda _{k}>0$ and the long run value of the
treatment rate $\bar{k}_{1}$. It is well-known that $K_{1}$ is Gaussian,
with variance equal to:
\begin{equation*}
\mathbb Var[K_{1}(t)]=\frac{\sigma _{k}^{2}}{2\lambda _{k}}(1-e^{-2\lambda _{k}t})
\end{equation*}
Thus, if mean-reversion is large compared to volatility $\sigma _{k}$,
constraint (\ref{positivity}) is satisfied. We simplify (\ref{dR/dt}) by:
\begin{equation*}
\frac{\frac{dR(t)}{dt}}{I(t)}=K_{0}+\mu _{0}+\alpha (t)(-K_{0}+K_{1}(t)-\mu
_{0}+\mu _{1})-\alpha (t)\sigma \frac{dB_{2}}{dt}  \label{Recovery_rate}
\end{equation*}
Putting everything together, the dynamics of the infected is:
\begin{equation*}
\frac{dI(t)}{dt}=\beta S(t)I(t)-\frac{dR(t)}{dt}-\sigma _{S}\sqrt{S(t)I(t)}%
\frac{dB_{1}}{dt}
\end{equation*}

We try to minimize a measure of the infected over our horizon $T$. To model
risk-aversion to unfavorable treatment decisions, the decision-maker is
supposed to minimize the expected value of a convex and increasing function
of $I(T)$. Alternately, one can maximize the negative thereof, i.e.,
maximize the expected value of a concave and decreasing function of $I(T).$
Such a function $U$ is called a \textit{utility function} in financial
economics. The policy obtained in maximizing the expected value of a concave
utility function can be showed, under certain conditions, to maximize the
expected value of the outcome (here $-I$)\ under a constraint on the
dispersion of the outcome. Out of the universe of concave decreasing utility
functions, we choose the power utility function:
\begin{equation*}
U(I)=-\frac{I^{1-\gamma }}{1-\gamma }
\end{equation*}

The coefficient $\gamma $ is often called the \textit{risk-aversion parameter%
}. When $\gamma =0$ the decision-maker is \textit{risk-neutral}, meaning
that the uncertainty does not have an influence on her decisions. It is
straightforward to check that this power utility function is concave in $I$
when $\gamma <0,$which we will assume. Taking for instance $\gamma =-1$, we
see that the objective is to%
\begin{equation*}
\max\mathbb E\big[-\frac{I^{2}}{2}\big]
\end{equation*}
which returns the same policy as:%
\begin{equation*}
\min\mathbb E\big[\frac{I^{2}}{2}\big]
\end{equation*}

The importance of analytic formulations is that other figures of interest in
this model, like the expected number of deaths from treatment can be
analytically calculated, and depend on $\gamma $. Thus, a decision-maker can
calibrate its risk-aversion parameter $\gamma $ on other goals. Expected
number of deaths is only one type of goal and economic factors that can be
easily added. We define 
\begin{equation*}
\tau =\min \{t>0|I(t)\leq 0\text{ or }I(t)\geq 1\}
\end{equation*}
Our controlled SIR\ model is thus:
\begin{eqnarray}
&&\hspace{2cm}\sup_{0\leq \alpha (t)\leq 1}\mathbb E\big[-\frac{I(\min (\tau ,T))^{1-\gamma }}{%
1-\gamma }\big]\notag\\
&&dS(t) =-\beta S(t)I(t)dt+\sigma _{S}\sqrt{S(t)I(t)}dB_{1}(t)
\label{susceptible} \\
&&dI(t) =\left( \beta S(t)-(K_{0}+\mu _{0})+\alpha (t)(K_{0}-K_{1}(t)+\mu
_{0}-\mu _{1})\right) I(t)dt  \notag \\
&&~~~~~~~~~+\alpha (t)I(t)\sigma dB_{2}(t)-\sigma _{S}\sqrt{S(t)I(t)}dB_{1}(t)  \label{infected}
\\
&&dK(t) =\lambda _{k}(\bar{K}-K(t))dt+\sigma _{k}dB_{2}(t)  \label{Kequ}
\end{eqnarray}

The relative sign of our volatilities $\sigma $ and $%
\sigma _{k}$ is important. We will assume without loss of generality that $%
\sigma <0$. The sign of $\sigma _{k}$ is the sign of covariance between the
measured value of today's treatment rate and the change in value of the
treatment rate between today and a future date. An example may help
illustrate the difference. Suppose that over a week one performs daily
measurements of the treatment recovery rate as well as daily forecasts of
the evolution of the treatment recovery rate over the next day. The two
quantities measured each day $t$ are proportional to the same white noise $%
B_{2}(t+$1 day$)-B_{2}(t)$. One then calculates weekly estimates $\hat{\sigma%
}$ of $\sigma $ and $\hat{\sigma}_{k}$ of $\sigma _{k}$ over these 7 daily
observations. Since we arbitrarily choose $\sigma >0$, a positive $\hat{%
\sigma}_{k}$ shows a correlation of +1 between the measurement (of today's
treatment rate) and the forecast. Figure \ref{fig:workflow} is a depiction of our model.

\begin{figure}[!htb]
    \centering
    \includegraphics[width = \textwidth]{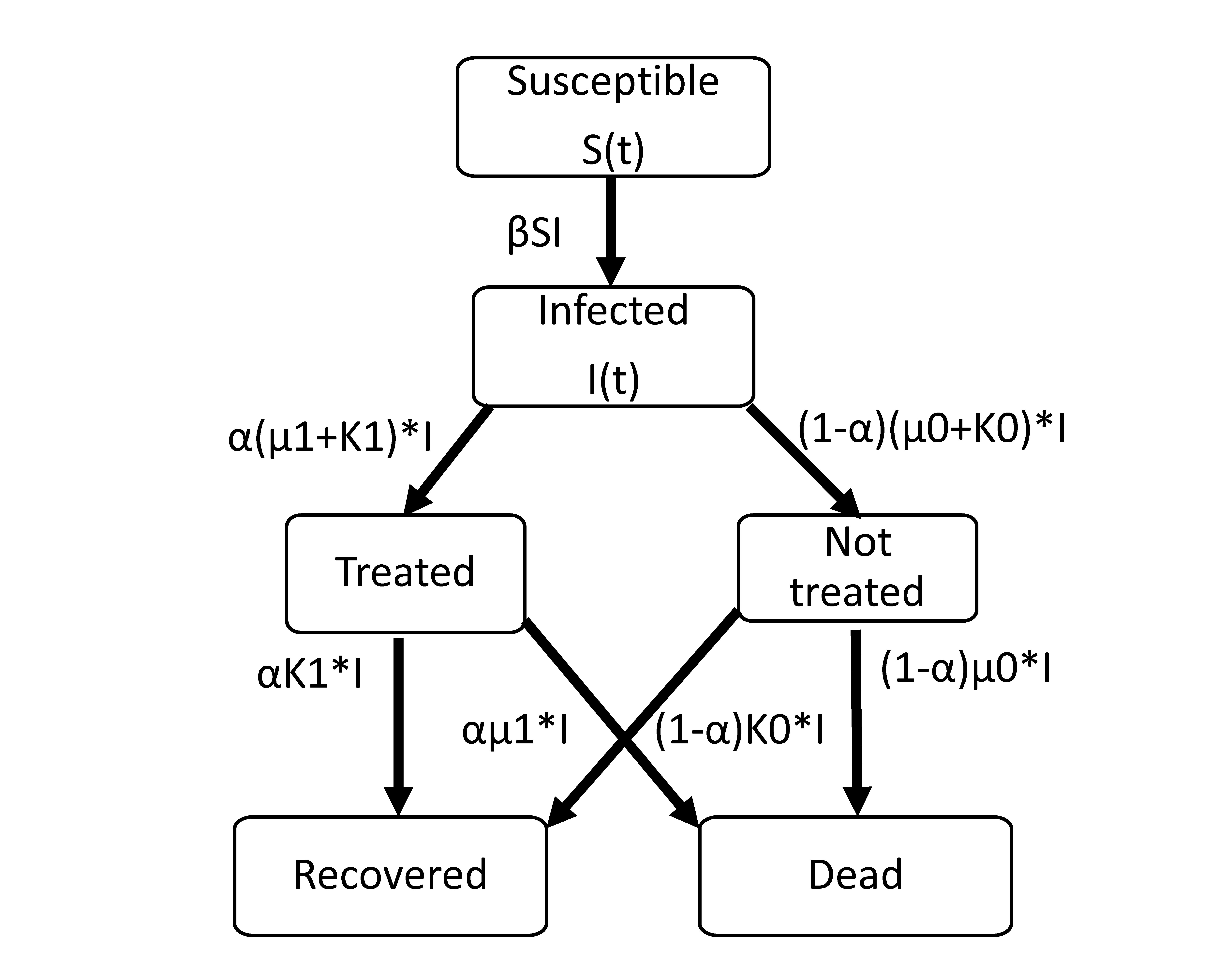}
    \caption{A stochastic SIR model}
    \label{fig:workflow}
\end{figure}

\begin{theorem}
\label{thm::uniqueness}
For any given intial value $(S(0),I(0),R(0))$ and any $%
X(0)$ there exists a unique solution of (\ref{susceptible})(\ref{infected})(%
\ref{Kequ}) up to time $\tau $.
\end{theorem}
The proof of Theorem \ref{thm::uniqueness}, included  in Appendix \ref{app::proof_uniquenss_thm},  follows the proof of a theorem of Yamada and Watanabe (1971), as exposed
in the book by \citet[Prop.~2.13, Sec.~5.2]{karatzas2014brownian}.

\section{Results in the Low Infection Regime}
\label{sec::low}
We assume $S(t)$ close to one and $\sigma _{S}=0$. Thus the term:%
\begin{equation*}
r=\beta S(t)-(K_{0}+\mu _{0})\simeq \beta -K_{0}-\mu _{0}
\end{equation*}
is assumed constant. With this simplification, we give an analytical
solution to the constrained problem, i.e., the case where $0\leq \alpha
(t)\leq 1$, a significant improvement over Gatto and Schellhorn, who
considered the unconstrained case.

We define the impact of treatment risk $X$:
\begin{equation*}
X(t)=\frac{K_{0}+\mu _{0}-\mu _{1}-K_{1}(t)}{\sigma }
\end{equation*}
as well as the long run impact of the treatment risk $\bar{X}$:
\begin{equation*}
\bar{X}=\frac{K_{0}+\mu _{0}-\mu _{1}-\bar{k}_{1}}{\sigma }
\end{equation*}

We define $\lambda _{x}=\lambda _{k}$ and $\sigma _{x}=\sigma _{k}/\sigma $.
For simplicity we write $\mu =K_{0}+\mu _{0}$. We consider first the case
where the treatment rate is constant, and then the case where it follows an
Ornstein Uhlenbeck process.
\subsection{Constant Treatment Rate}

Let $b=\beta-\mu_1-\bar{k}_{1}$. The problem is:

\begin{equation}
\label{ConstantTreatment}
\begin{gathered}
\sup_{0\leq \alpha (t)\leq 1}\mathbb E\big[-\frac{I(T)^{1-\gamma }}{1-\gamma }\big] \\
dI(t)=(r+\alpha (t)(b-r))I(t)dt+\alpha (t)\sigma I(t)dB_{2}(t)
\end{gathered}
\end{equation}

\begin{theorem}
\label{thm::optimal_control_low_regime}
The optimal control is constant, and satisfies

\begin{equation*}
\alpha =\min \Big(1,\max \big(0,\frac{\bar{k}_1-K_0}{\sigma ^{2}|\gamma |}\big)\Big)
\end{equation*}
\end{theorem}
The proof is the Appendix \ref{proof::optimal_control_low_regime}, and follows closely \cite{cvitanic1992convex}.

\subsection{Treatment Rate as Ornstein-Uhlenbeck Process}

The problem is
\begin{gather}
\label{OU_low_infect_prob}
\begin{aligned}
    &\hspace{2cm}\sup \mathbb E\big[-\frac{I(T)^{1-\gamma }}{1-\gamma }\big] \\
&dI(t)=(r+\alpha (t)\sigma X(t))I(t)dt+\alpha (t)\sigma I(t)dB_{2}(t) \\
&dX(t)=\lambda _{x}(\bar{X}-X(t))dt-\sigma _{x}dB_{2}(t)
\end{aligned}
\end{gather}

In the low infection regime our solution will depend on a kernel $H_{0}(X_{t},\tau)$ with $\tau = T-t$, while
in the moderate infection regime it will also depend on two other kernels $%
H_{1}(X_{t},\tau) $ and $H_{2}(X_{t},\tau)$ that are closely related. In order to
unify notation we define the kernels. Define
\begin{equation}
\label{kernal_H0}
H_{0}(X_{t},\tau)=\exp \left( \frac{1}{\gamma }\Big(\frac{A_{1}(\tau ,\gamma
)X_{t}^{2}}{2}+A_{2}(\tau ,\gamma )X_{t}+A_{3}(\tau ,\gamma )+(1-\gamma
)(\mu +r)\tau\Big) \right)
\end{equation}%
and, for $i>0$
\begin{equation}
H_{i}(X_{t},\tau)=\exp \left( \frac{i}{\gamma }\Big(\frac{A_{1}(\tau ,\gamma
/i)X_{t}^{2}}{2}+A_{2}(\tau ,\gamma /i)X_{t}+A_{3}(\tau ,\gamma /i)\Big)\right)
\label{KernelH1}
\end{equation}
where
\begin{eqnarray}
&&A_{1}(\tau ,\gamma ) =\frac{1-\gamma }{\gamma }\frac{2(1-\exp (-\theta
(\gamma )\tau ))}{2\theta (\gamma )-(b_{2}(\gamma )+\theta (\gamma ))(1-\exp
(-\theta (\gamma )\tau ))}  \label{A1} \\
&&A_{2}(\tau ,\gamma ) =\frac{4\lambda _{x}\bar{X}b_{1}(\gamma )\left(
1-\exp \left( -\theta (\gamma )\tau /2\right) \right) ^{2}}{\theta (\gamma
)\left( 2\theta (\gamma )-(\theta (\gamma )+b_{2}(\gamma ))(1-\exp (-\theta
(\gamma )\tau ))\right) } \label{A2} \\
&&A_{3}(\tau ,\gamma ) =\int_{0}^{\tau }\left( \frac{\sigma _{x}^{2}}{2\gamma 
}+\lambda _{x}\bar{X}\right) A_{2}^{2}(s,\gamma )+\frac{\sigma _{x}^{2}}{2}%
A_{1}(s,\gamma )+(\gamma -1)\mu ds  \label{A3}\\
&&b_{1}(\gamma ) =\frac{1-\gamma }{\gamma } ~~~~~b_{2}(\gamma ) =2\big(\frac{\gamma -1}{\gamma }\sigma _{x}-\lambda _{x}\big) ~~~~~b_{3}(\gamma ) =\frac{\sigma _{x}^{2}}{\gamma } \notag \\
&&\theta (\gamma ) =\sqrt{b_{2}^{2}(\gamma )-4b_{1}(\gamma )b_{3}(\gamma )}\notag
\end{eqnarray}

We provide an explicit formula for $A_{3}(\tau ,\gamma ) $ in Appendix \ref{app::explicit_A3}.
The solution to (\ref{OU_low_infect_prob}) is shown in \citet[Prop. 1]{gatto2021optimal}, but with some typos in the expression of $H_0(X_{t},\tau)$, hence we correct the mistake by providing (\ref{kernal_H0}) in this paper.

\section{Results in the Moderate Infection Regime}
\label{sec::moderate}
We first handle the Ornstein-Uhlenbeck treatment rate case, which was presented in \citet[Prop. 2]{gatto2021optimal}. In this work, we aim to correct the typos and provide more detials for the  Proposition 2 in \cite{gatto2021optimal}.

\subsection{Treatment Rate as Ornstein-Uhlenbeck Process}

The problem is defined in Section 2. We rewrite here for convenience,
\begin{gather}
\label{ou_moderate_prob}
    \begin{aligned}
    &\hspace{2cm}\sup\mathbb E\big[-\frac{I(\min (\tau ,T))^{1-\gamma }}{%
    1-\gamma }\big] \\
    &dS(t) =-\beta S(t)I(t)dt+\sigma _{S}\sqrt{S(t)I(t)}dB_{1}(t) \\
    &dI(t) =\left( \beta S(t)-\mu +\alpha (t)\sigma X(t)\right) I(t)dt +\alpha (t)I(t)\sigma dB_{2}(t)\\
    &\hspace{1cm}-\sigma _{S}\sqrt{S(t)I(t)}dB_{1}(t)
    \\
    &dX(t) =\lambda _{x}(\bar{X}-X(t))dt-\sigma _{x}dB_{2}(t)
    \end{aligned}
\end{gather}
To express the solution of (\ref{ou_moderate_prob}), we further define 
\begin{eqnarray}
&&\tilde{M}(t,\tau ) =\frac{2\theta (\gamma )e^{\frac{1}{2}(b_{2}(\gamma
/2)-b_{2}(\gamma)-\theta(\gamma ))(\tau -t)}}{2\theta (\gamma )-(b_2(\gamma )+\theta(\gamma ))\left(
1-e^{-\theta (\gamma )(\tau -t)}\right) }\label{M_tilde} \\
&&m_{Y}(\tau ,x) =x\tilde{M}(t,\tau )+\int_{s=t}^{\tau }\tilde{M}(s,\tau
)(\lambda _{x}\bar{X}+\frac{\sigma _{x}^{2}}{^{\gamma }}A_{2}(\tau -s,\gamma
))ds\notag\\
&&\hspace{2cm}+\frac{A_{2}(T-\tau ,\gamma )}{A_{1}(T-\tau ,\gamma )} \notag\\
&&V_{Y}(\tau ,x) =\sigma _{x}^{2}\int_{t}^{\tau }\tilde{M}^{2}(s,\tau )ds \notag\\
&&g(X,t) =\int_{\tau =t}^{T}H_{2}(X,\tau -t)\frac{1}{2}\frac{\beta ^{2}}{%
\sigma _{S}^{2}\gamma }\frac{1}{\sqrt{1-2V_{Y}(\tau ,X)A_{1}(T-\tau ,\gamma
)/\gamma }} \notag\\
&&\hspace{1.5cm}\times\exp \Big( \frac{2}{\gamma }A_{3}(T-\tau ,\gamma )-\frac{A_{2}^{2}(T-\tau
,\gamma )}{\gamma A_{1}(T-\tau ,\gamma )}\label{func_g}\\
&&\hspace{3cm}+\frac{m_{Y}^{2}(\tau
,X)A_{1}(T-\tau ,\gamma )}{\gamma -2V_{Y}(\tau ,X)A_{1}(T-\tau ,\gamma )}%
\Big) d\tau\notag
\end{eqnarray}
From this, we can calculate:
\begin{eqnarray*}
&&\frac{\partial g}{\partial X} 
=\int_{\tau =t}^{T}H_{2}(X,\tau -t)\frac{1}{2}\frac{\beta ^{2}}{\sigma
_{S}^{2}\gamma }\frac{1}{\sqrt{1-2V_{Y}(\tau ,X)A_{1}(T-\tau ,\gamma
)/\gamma }} \\
&&\times\exp \left( \frac{2}{\gamma }A_{3}(T-\tau,\gamma )-\frac{A_{2}^{2}(T-\tau ,\gamma
)}{\gamma A_{1}(T-\tau ,\gamma )}+\frac{m_{Y}^{2}(\tau ,X)A_{1}(T-\tau
,\gamma )}{\gamma -2V_{Y}(\tau ,X)A_{1}(T-\tau ,\gamma )}\right)  \notag \\
&&\times\Big( \frac{A_{1}(\tau -t,\gamma /2)X(t)+A_{2}(\tau -t,\gamma /2)}{%
\gamma /2}\\
&&\hspace{3cm}+\frac{2m_{Y}(\tau ,X)A_{1}(T-\tau ,\gamma )}{\gamma -2V_{Y}(\tau
,X)A_{1}(T-\tau ,\gamma )}\tilde{M}(t,\tau )\Big) d\tau 
\end{eqnarray*}
\begin{theorem}
\label{thm::ou_moderate_prob}
Let $I(0)=\varepsilon $. Suppose $I(t)\leq 1$. If $\sigma _{x}<0$ then the problem (\ref{ou_moderate_prob}) has a solution such that
\begin{equation*}
I(t) =\varepsilon Z^{1/\gamma }(t)H_{1}(X(t),T-t)+\varepsilon^{2}Z^{2/\gamma
}(t)S(t)g(X(t),t)+O(\varepsilon ^{3}) 
\end{equation*}
where $Z(t)$ satisfies:
\begin{eqnarray}
&&\frac{dZ}{Z} =(-\mu +X^{2}+\frac{\beta ^{2}SI}{\sigma _{S}^{2}})dt-\frac{%
\beta \sqrt{SI}}{\sigma _{S}}dB_{1}+XdB_{2} \notag \\
&&Z(0) =\left(\frac{-H_1(X(0),T)+\sqrt{H_1^2(X(0),T)-4\varepsilon S(0)g(X(0), 0)(\mathcal O(\varepsilon^2)-1)}}{2\varepsilon S(0)g(X(0), 0)}\right)^{\gamma} \notag
\end{eqnarray}
The optimal control $\alpha ^{\ast }(t)=\alpha
_{0}(t)+\varepsilon \alpha _{1}(t)+\mathcal O(\varepsilon^2)$ where:
\begin{eqnarray*}
&&\alpha _{0}(t)=\frac{X(t)}{\gamma \sigma }-\frac{\sigma _{x}}{\gamma \sigma }%
\left( A_{1}(T-t,\gamma )X(t)+A_{2}(T-t,\gamma )\right)\\
&&\alpha _{1}(t) =\frac{Z^{1/\gamma }(t)S(t)}{H_{1}(X(t),T-t)\sigma }\Big( 
\frac{g(X(t),t)X(t)}{\gamma }-\sigma _{x}\frac{\partial g}{\partial X}
\\
&&\hspace{1cm} +\sigma _{x}\frac{g(X(t),t)}{\gamma }\left( A_{1}(T-t,\gamma
)X(t)+A_{2}(T-t,\gamma) \right)\Big)
\end{eqnarray*}
\end{theorem}
The proof is in Appendix \ref{app::proof_thm3}. We refer to \cite{gatto2021optimal} for a discussion of  $\alpha_0$. The sign of $\alpha_1$ is determined by the signs of  $\sigma$ and
\begin{equation}
\label{alpha_1_part2}
    \frac{g(X(t),t)}{\gamma}\Big(X(t)+\sigma_x\left(A_{1}(T-t,\gamma)X(t)+A_{2}(T-t,\gamma)\right)\Big)-\sigma_x\frac{\partial g}{\partial X}
\end{equation}
More specifically, $\alpha_1$ is positive if $\sigma$ and (\ref{alpha_1_part2}) are both positive or negative.  $\alpha_1$ is negative if one of them is positive and the other one is negative.

It is  obvious that the magnitude of both $g(X(t),t)$ and $\frac{\partial g}{\partial X}$ decrease with time and are equal to zero when $t=T$. Therefore, the importance of $\alpha_1$ decreases as time increases. 

To further discuss the sign of (\ref{alpha_1_part2}), we rewrite it by
$$
\left|\frac{g(X(t),t)}{\gamma}\right|\left(\left(1+\left|\sigma_x A_{1}(T-t,\gamma)\right|\right)X(t)+\bar{X}\left|\frac{\sigma_xA_{2}(T-t,\gamma)}{\bar{X}}\right|\right)+|\sigma_x|\frac{\partial g}{\partial X}
$$
Thus, suppose $\frac{\partial g}{\partial X}$, $X(t)$, and $\bar{X}$ are all positive, (\ref{alpha_1_part2}) is positive, and vice versa. In the following  cases, we provide two simple cases that we can easily discuss the sign of $\alpha_1$:
\begin{itemize}
    \item if $\frac{\partial g}{\partial X}\sigma>0$, $\mu-\mu_1>\max(K_1(t),\bar{k}_1)$, then  $\alpha_1$ is positive.
    \item if $\frac{\partial g}{\partial X}\sigma<0$, $\mu-\mu_1<\min(K_1(t),\bar{k}_1)$,  then  $\alpha_1$ is negative.
\end{itemize}

In the following, we discuss the full expansion of the solution in Theorem \ref{thm::ou_moderate_prob}. Consider equation (57) in \cite{gatto2021optimal}:
\begin{equation*}
\Big(\frac{\partial }{\partial t}+L_{1}+\varepsilon L_{2}\Big)f=0 
\end{equation*}
This time we use full asymptotic expansion:%
\begin{equation*}
f=f_{1}+\varepsilon f_{2}+\ldots=\sum_{i=1}^\infty f_i \varepsilon^{i-1}
\end{equation*}
and obtain:
\begin{eqnarray*}
0
&=&\Big(\frac{\partial }{\partial t}+L_{1}\Big)f_{1}+\sum_{i=1}^\infty\Big(\Big(\frac{\partial }{\partial t}+L_{1}\Big)f_{i+1}+L_2f_i\Big)\varepsilon^i
\end{eqnarray*}
The terms of our asymptotic expansion are thus determined by:%
\begin{eqnarray}
&&\Big( \frac{\partial }{\partial t}+L_{1}\Big) f_{1} =0  \label{first_eq_full_expansion}
\\
&&\Big(\frac{\partial }{\partial t}+L_{1}\Big)f_{i+1} =-L_{2}f_{i} ~~~~~~i=1, 2, \ldots \label{second_eq_full_expansion}
\end{eqnarray}
We use the Ansatz: 
$$
f_i(Z(t),  X(t), t) = Z(t)^{2^{i-1}/\gamma}S(t)^{2^{i-1}-1}g_i(X(t), t)~~~~~~i=1, 2, \ldots
$$
We have showed that $g_1 =H_1 $ and $g_2=g$ in the proof of Theorem \ref{thm::ou_moderate_prob}. By the same process, we can also calculate the expressions for $g_3$, $g_4, \ldots$ in the sequel. 
\subsection{Constant Treatment Rate}
The problem is:
\begin{gather}
\label{constant_moderate_prob}
\begin{aligned}
&\hspace{2cm}\sup \mathbb E\big[-\frac{I(T)^{1-\gamma }}{1-\gamma }\big] \\
&dS(t)=-\beta S(t)I(t)dt+\sigma _{S}\sqrt{S(t)I(t)}dB_{1}(t) \\
&dI(t)=(r+\alpha (t)(b-r))I(t)dt+\alpha (t)\sigma I(t)dB_{2}(t)-\sigma _{S}%
\sqrt{S(t)I(t)}dB_{1}(t)
\end{aligned}
\end{gather}
Let $\tau = T-t$, the solution kernels $h_{i}(\tau)$ for $i=1,2,\ldots$ are given by:
\begin{equation}
h_{i}(\tau)=\exp \Big( \frac{2^{i-1}}{\gamma }\Big(\frac{a_{i,1} }{2%
}\big(\frac{b-r}{\sigma}\big)^2+a_{i,2} \Big)\tau \Big)  \label{Kernelh}
\end{equation}
where
\begin{eqnarray}
a_{i,1} &=&\frac{1-\gamma/2^{i-1} }{\gamma/2^{i-1} }\label{ai1}\\
a_{i,2} &=&(\gamma/2^{i-1}-1)\mu  \label{ai2}
\end{eqnarray}

\begin{theorem}
\label{thm::constant_moderate}
Let $I(0)=\varepsilon $, then the problem (\ref{constant_moderate_prob}) has a
solution such that
\begin{eqnarray*}
I(t) =\sum_{i=1}^\infty Z(t)^{2^{i-1}/\gamma}S(t)^{2^{i-1}-1}g_i( t)\varepsilon^i
\end{eqnarray*}%
where $g_1(t)=h_1(T-t)$, $g_i(t)$, $i>1$ can be obtained by (\ref{g_i}), and $Z(t)$ satisfies:
\begin{eqnarray*}
&&\frac{dZ}{Z} =\Big(-\mu +\big(\frac{b-r}{\sigma}\big)^2+\frac{\beta ^{2}SI}{\sigma _{S}^{2}}\Big)dt-\frac{%
\beta \sqrt{SI}}{\sigma _{S}}dB_{1}+\frac{b-r}{\sigma}dB_{2}
\label{dz_secondmodel} \\
&&1 =\sum_{i=1}^\infty Z(0)^{2^{i-1}/\gamma}S(0)^{2^{i-1}-1}g_i( 0)\varepsilon^{i-1}
\end{eqnarray*}
Moreover the optimal proportion undergoing treatment  $\alpha ^{\ast
}(t)$ equal to $\alpha _{0}(t)+\varepsilon \alpha
_{1}(t)+\mathcal O(\varepsilon ^{2})$, where $\alpha _{0}(t)$ and $\alpha _{1}(t)$ are
equal to
\begin{equation*}
\alpha_0=\frac{b-r}{\gamma\sigma^2}
~~~~~~~~
\alpha_1 =\frac{Z^{1/\gamma}(t)S(t)}{h_1(T-t)g_2(t)}\frac{b-r}{\gamma\sigma^2}
\end{equation*}
\end{theorem}
The proof is in Appendix \ref{app::proof_thm4}, where we also provide a formula for $g_3$. Observe that $$
g_2(t)=\frac{\beta ^{2}}{2\sigma _{S}^{2}}\frac{h_2(T-t)-h_1^2(T-t)}{\big(\frac{b-r}{\sigma}\big)^2(a_{1,1}-a_{2,1})+2(a_{1,2}-a_{2,2})} = \frac{\beta ^{2}}{2\sigma _{S}^{2}}\frac{h_2(T-t)-h_1^2(T-t)}{\gamma\mu-\big(\frac{b-r}{\sigma}\big)^2/\gamma}
$$ is always positive because the signs of $h_2(T-t)-h_1^2(T-t)$ and $\gamma\mu-\big(\frac{b-r}{\sigma}\big)^2/\gamma$ are the same. The signs of $\alpha_0$ and $\alpha_1$ are determined by the sign of $\frac{r-b}{\sigma^2}$.

\section{Application to COVID-19}
\label{sec::application}
We use the same data set and parameters (see Table \ref{tab:parameters}) as in \cite{gatto2021optimal}, but this time we show the optimal control (result of Theorem \ref{thm::optimal_control_low_regime}) of problem (\ref{ConstantTreatment}).
\begin{table}[!htb]
    \centering
    \begin{tabular}{|l|l|l|}
     \hline
Treatment Parameter & Symbol & Value \\ \hline\hline
Death rate/no treatment & $\mu _{0}$ & 0.0575 \\ \hline
Death rate & $\mu _{1}$ & 0.0575 \\ \hline
Recovery rate/ no treatment & $K_{0}$ & 0.2559 \\ \hline
Recovery rate at time 0 & $K_{1}(0)$ & 0.2559 \\ \hline
Long run value of recovery rate & $\bar{k}_{1}$ & 0.4612 \\ \hline
Volatility of the measurement of today's recovery rate & $\sigma $ & 0.4418 \\ 
\hline
Volatility of changes in the recovery rate & $\sigma _{k}$ & -1.1647 \\ \hline
Speed of mean-reversion of the recovery rate & $\lambda _{k}$ & 0.7692 \\
\hline
Transmission rate & $\beta$ & 0.025\\
\hline
Proportion  of infected at time 0 & $\varepsilon$ &0.01\\
 \hline
 Time step &$\Delta t$ & 0.001\\
\hline
    \end{tabular}
    \caption{Parameters}
    \label{tab:parameters}
\end{table}
We compare in Figure \ref{fig:simulation_plot} three types of treatment:
\begin{itemize}
\setlength\itemsep{-0.1em}
    \item no treatment
    \item full control, i.e., $\alpha(t)=1$
    \item optimal control, given in  Theorem \ref{thm::optimal_control_low_regime}
\end{itemize}
We can see that, for all risk-aversion parameters $\gamma$ considered (between $-1$ and $-5$), our control is better.
\begin{figure}[!htb]
    \centering
    \includegraphics[width=\textwidth]{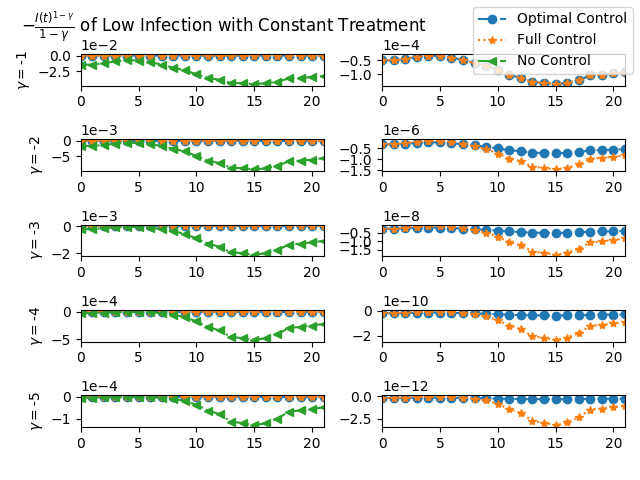}
    \caption{Optimal control of low infection with constant treatment. Weekly US COVID-19 data from June 7, 2020 to November 1, 2020. Github  repository for generating the plot: \href{https://github.com/yujiading/optimal-control-sir-model}{https://github.com/yujiading/optimal-control-sir-model}. }
    \label{fig:simulation_plot}
\end{figure}
\section{Conclusion}
\label{sec::conclusion}
We showed that a stochastic optimal control approach enables to fight the COVID-19 epidemic better. Many interesting problems remain to be solved. For instance, we could analytic constrained policies in the multiple treatment case or the Ornstein-Uhlenbeck case. Optimal vaccination is another area where we believe a similar asymptotic approach can be used. Finally, \cite{bertozzi2020challenges} use Hawkes processes to model COVID-19. The control of Hawkes processes remains a largely open problem that deserves attention, in particular for its application to epidemiology.

\begin{appendices}

\section{Proof of Theorem \ref{thm::uniqueness}}
\label{app::proof_uniquenss_thm}
We follow the proof in \citet[Prop.~2.13, Sec.~5.2]{karatzas2014brownian}. They consider the one-dimensional case.
Let $h:[0,\infty )\rightarrow \lbrack 0,\infty )$ be a strictly increasing
function with $h(0)$ and 
\begin{equation}
\int_{(0,\varepsilon )}h^{-2}(u)=\infty \text{, \ \ }\forall \varepsilon >0
\label{bof}
\end{equation}
In our case, we take $h(x)=x$. \ Because of (\ref{bof}), there exists a
strictly decreasing sequence $\{a_{n}\}\subset (0,1]$ with $a_{0}$ and $%
\lim_{n\rightarrow \infty }a_{n}=0$ such that $%
\int_{a_{n}}^{a_{n-1}}h^{-2}(u)du=n$. For every $n$ there exists continuous
function $\rho _{n}$ on $\mathbb{R}$ with support on $(a_{n},a_{n-1})$ so
that
\begin{equation*}
0\leq \rho _{n}(x)\leq \frac{2}{nh(x)}\text{, \ \ }x>0
\end{equation*}
and $\int_{a_{n-1}}^{a_{n}}\rho _{n}(u)du=1$. Then the function 
\begin{equation*}
\Psi _{n}(x)=\int_{0}^{|x|}\int_{0}^{y}\rho _{n}(u)dudy
\end{equation*}
is even and twice continuously differentiable, with $|\Psi _{n}^{\prime
}(x)|\leq 1$ and $\lim_{n\rightarrow \infty }$ $\Psi_{n}(x)$ $=|x|$. Suppose
there are two strong solutions $(I^{(1)},S^{(1)})$ and $(I^{(2)},S^{(2)})$,
\begin{eqnarray*}
&&d(I^{(1)}-I^{(2)}-\mathbb E[I^{(1)}-I^{(2)}])-\alpha \sigma (I^{(1)}-I^{(2)})dB_{2}
\\
&&~~~=-\sigma _{S}\left( \sqrt{S^{(1)}I^{(1)}}-\sqrt{S^{(2)}I^{(2)}}\right) dB_{1}\\
&&~~~=
-d(S^{(1)}-S^{(2)}-\mathbb E[S^{(1)}-S^{(2)}])
\end{eqnarray*}
so that
\begin{eqnarray*}
&&(d(I^{(1)}-I^{(2)}))^{2} <\sigma _{S}^{2}(
S^{(1)}I^{(1)}-S^{(2)}I^{(2)}) dt+(\alpha \sigma
)^{2}(I^{(1)}-I^{(2)})^{2}dt \\
&&(d(S^{(1)}-S^{(2)}))^{2} <\sigma _{S}^{2}(
S^{(1)}I^{(1)}-S^{(2)}I^{(2)}) dt \\
&&d((I^{(1)}-I^{(2)})(S^{(1)}-S^{(2)})) <-\sigma _{S}^{2}(
S^{(1)}I^{(1)}-S^{(2)}I^{(2)}) dt
\end{eqnarray*}
Thus, since $|\Psi _{n}^{\prime }|<1$,
\begin{eqnarray*}
&&\mathbb E[d\Psi _{n}(I_{t}^{(1)}-I_{t}^{(2)})+d\Psi _{n}(S_{t}^{(1)}-S_{t}^{(2)})]
\\
&&~~~=\mathbb E[\Psi _{n}^{\prime }(I_{t}^{(1)}-I_{t}^{(2)})(d(I^{(1)}-I^{(2)}))+\Psi
_{n}^{\prime }(S_{t}^{(1)}-S_{t}^{(2)})(d(S^{(1)}-S^{(2)}))] \\
&&~~~~~~+\frac{1}{2}\mathbb E[\Psi _{n}^{\prime \prime
}(I_{t}^{(1)}-I_{t}^{(2)})(d(I^{(1)}-I^{(2)}))^{2}]\\
&&~~~~~~+\frac{1}{2}\mathbb E[\Psi _{n}^{\prime \prime
}(S_{t}^{(1)}-S_{t}^{(2)})(d(S^{(1)}-S^{(2)}))^{2}] \\
&&~~~\leq \mathbb E[2|\beta||S^{(1)}I^{(1)}-S^{(2)}I^{(2)}|dt+|D||I^{(1)}-I^{(2)}|dt] \\
&&~~~~~~+\frac{1}{2}\mathbb E[\Psi _{n}^{\prime \prime }(I_{t}^{(1)}-I_{t}^{(2)})\sigma _{S}^{2}(
S^{(1)}I^{(1)}-S^{(2)}I^{(2)}) dt]\\
&&~~~~~~+\frac{1}{2}\mathbb E[\Psi _{n}^{\prime \prime }(S_{t}^{(1)}-S_{t}^{(2)})\sigma _{S}^{2}(
S^{(1)}I^{(1)}-S^{(2)}I^{(2)}) dt] \\
&&~~~~~~+\frac{1}{2}\mathbb E[\Psi _{n}^{\prime \prime }(I_{t}^{(1)}-I_{t}^{(2)})(\alpha
\sigma )^{2}(I^{(1)}-I^{(2)})^{2}]dt
\end{eqnarray*}
where
\begin{equation*}
D=-(K_{0}+\mu _{0})+\alpha (K_{0}-K_{1}+\mu _{0}-\mu _{1})
\end{equation*}
Observe that:
\begin{eqnarray*}
S^{(1)}I^{(1)}-S^{(2)}I^{(2)}
&=&S^{(1)}(I^{(1)}-I^{(2)})+I^{(2)}(S^{(1)}-S^{(2)}) \\
&<&|I^{(1)}-I^{(2)}|+|S^{(1)}-S^{(2)}|
\end{eqnarray*}
Since $\Psi _{n}^{\prime \prime }<2/nh$ and $h$ is positive,
\begin{eqnarray*}
&&\left( \Psi _{n}^{\prime \prime }(I_{t}^{(1)}-I_{t}^{(2)})+\Psi _{n}^{\prime
\prime }(S_{t}^{(1)}-S_{t}^{(2)})\right) \left(
S^{(1)}I^{(1)}-S^{(2)}I^{(2)}\right)  \\
&&~~~<\frac{2}{n}\left( \frac{1}{h(|I^{(1)}-I^{(2)}|)}+\frac{1}{%
h(|S^{(1)}-S^{(2)}|)}\right) (|I^{(1)}-I^{(2)}|+|S^{(1)}-S^{(2)}|) \\
&&~~~<\frac{2}{n}\left( \frac{|I^{(1)}-I^{(2)}|}{h(|I^{(1)}-I^{(2)}|)}+\frac{%
|S^{(1)}-S^{(2)}|}{h(|S^{(1)}-S^{(2)}|)}\right)
\end{eqnarray*}%
Taking $h(x)=x$ results in
\begin{eqnarray*}
&&\mathbb E[d\Psi _{n}(I_{t}^{(1)}-I_{t}^{(2)})+d\Psi
_{n}(S_{t}^{(1)}-S_{t}^{(2)})]\\
&&~~~<\Big( \mathbb E[(2|\beta|+|D|)|I^{(1)}-I^{(2)}|]+\mathbb E[2|\beta||S^{(1)}-S^{(2)}|]\\
&&~~~~~~+\frac{2\sigma_S^2}{n}%
+\frac{(\alpha \sigma )^{2}}{2}\mathbb E[|I^{(1)}-I^{(2)}|]\Big) dt
\end{eqnarray*}
Since $\lim_{n\rightarrow \infty }\Psi _{n}(x)=|x|$,
\begin{eqnarray*}
&&\mathbb E[|I_{t}^{(1)}-I_{t}^{(2)}|+|S_{t}^{(1)}-S_{t}^{(2)}|]\\
&&~~~<
\int_{0}^{t}\mathbb E[(2|\beta|+|D_{s}|)|I_{s}^{(1)}-I_{s}^{(2)}|]+\mathbb E[2|\beta||S_s^{(1)}-S_s^{(2)}|]\\
&&~~~~~~+\frac{(\alpha_s \sigma )^{2}}{2}\mathbb E[|I_s^{(1)}-I_s^{(2)}|]ds
\end{eqnarray*}
But,
$$
\mathbb E[(2|\beta|+|D_{s}|)|I_{s}^{(1)}-I_{s}^{(2)}|]
<\sqrt{\mathbb E[(2|\beta|+|D_{s}|)^{2}]}\sqrt{\mathbb E[|I_{s}^{(1)}-I_{s}^{(2)}|^{2}]}
$$
Since $|I_{s}^{(1)}-I_{s}^{(2)}|<1$, $\mathbb E[(I_{s}^{(1)}-I_{s}^{(2)})^{2}]<1$
and $\sqrt{\mathbb E[(I_{s}^{(1)}-I_{s}^{(2)})^{2}]}<\mathbb E[|I_{s}^{(1)}-I_{s}^{(2)}|]$ thus
\begin{eqnarray*}
&&\mathbb E[|I_{t}^{(1)}-I_{t}^{(2)}|+|S_{t}^{(1)}-S_{t}^{(2)}|]\\
&&~~~<\int_{0}^{t}\left( \sqrt{\mathbb E[(2|\beta|+|D_{s}|)^{2}]}%
+\frac{(\alpha \sigma )^{2}}{2}\right) \mathbb E[|I_{s}^{(1)}-I_{s}^{(2)}|]\\
&&~~~~~~+2|\beta|\mathbb E[|S_{s}^{(1)}-S_{s}^{(2)}|]ds\\
&&~~~<\int_{0}^{t}\max\left( \sqrt{\mathbb E[(2|\beta|+|D_{s}|)^{2}]}%
+\frac{(\alpha \sigma )^{2}}{2}, 2|\beta|\right)\\
&&~~~~~~\times\mathbb E[|I_{s}^{(1)}-I_{s}^{(2)}|+|S_{s}^{(1)}-S_{s}^{(2)}|]ds
\end{eqnarray*}
and local uniqueness follows by Gronwall's inequality.

\section{Proof of Theorem \ref{thm::optimal_control_low_regime}}
\label{proof::optimal_control_low_regime}
We refer to the problem treated by \cite{gatto2021optimal} as the \textit{%
unconstrained} problem. Indeed, in that problem $\alpha $ was not
constrained. We refer to our problem as the \textit{constrained} problem. We
follow the method of proof in \cite{cvitanic1992convex}, referred to hereafter as CK. They introduce 
\textit{auxiliar}y problems, which are unconstrained. They show that there
exists an auxiliary problem which solution can be used to construct the
solution of the original constrained problem. We follow the numbering of the
sections in CK in order to ease understanding.
\paragraph{CK Section 2. The Model.}
To ease the correspondence with the CK paper, we define $b-r=K_{0}+\mu _{0}-\mu _{1}-\bar{k}_{1}$, $\theta:=(b-r)/\sigma$, and
\begin{eqnarray*}
H^{(0)}(t) =\exp (-rt)\exp \Big(-\int_{0}^{t}\theta dB_2(s)-\frac{1}{2}%
\int_{0}^{t}\theta^{2}ds\Big) 
\end{eqnarray*}
Observe that
$\mathbb E[\int_{0}^{t}\theta ^{2}ds]<\infty$.
\paragraph{CK Section 3. Portfolio and consumption processes.}
Define:
\begin{equation*}
B_2^{(0)}(t)=B_2(t)+\int_{0}^{t}\theta ds
\end{equation*}

Denote by $I^{i,\alpha }$ the infected process subject to $I(0)=i$ and
control $\alpha $. It is admissible if 
\begin{equation*}
0\leq I^{i,\alpha }(t)\leq 1\text{ \ \ \ }\forall~ 0\leq t\leq T
\end{equation*}

The set of admissible $\alpha $ is denoted $\mathcal{A}_{0}(i)$. Note that
(See (3.5)\ in CK)
\begin{eqnarray*}
H^{(0)}(t)I(t) =i+\int_{0}^{t}H^{(0)}(s)I(s)(\alpha (s)\sigma -\theta)dB_{2}(s)
\end{eqnarray*}
\paragraph{CK Section 4. Convex sets and their support functions.}

The difference between CK and this paper is that our objective is to minimize. This means that the key relation between our
auxiliary infected and infected is reversed compared to the first equation
in CK. Indeed if $\alpha _{\nu }$ solves the auxiliary problem and $\alpha $
the original problem, we must have:
\begin{equation*}
I_{\nu }^{i,\alpha _{\nu }}(t)\leq I^{i,\alpha }(t)
\end{equation*}

Define 
\begin{equation}
\label{delta_function}
\delta (\nu )=\left\{ 
\begin{array}{cc}
0 & \nu <0 \\ 
\nu & \nu >0
\end{array}
\right.
\end{equation}
It is subadditive:%
\begin{equation}
\delta (\lambda +\nu )\leq \delta (\lambda )+\delta (\nu )
\label{superadditive}
\end{equation}
\paragraph{CK Section 5. Utility functions.}
The main difference between our utility functions and the utility functions
in financial economics is that our utility functions are decreasing for
positive arguments. Recall indeed that our utility function is, for $\gamma<0$:
$$
U(i) = -\frac{i^{1-\gamma}}{1-\gamma}
$$
Since 
$$
U'(i) = -i^{-\gamma}
$$
We have $\lim_{i\rightarrow \infty }U'(i)=-\infty$ and $\lim_{i\rightarrow 0}U'(i)=0$, again for $\gamma<0$. This is  unlike CK and \cite{wachter2002portfolio} who consider the case $0<\gamma<1$ with utility of wealth $U_2(x)=\frac{x^{1-\gamma}}{1-\gamma}$. In their case, $\lim_{x\rightarrow \infty }U_2'(x)=0$ and $\lim_{x\rightarrow 0}U_2'(x)=\infty$.

We define $I_{2}$ to be the inverse of $U^{\prime }$, with $I_{2}(y)$ on $y\leq0$. By straighforward calculations:
$$
I_{2}(y) = (-y)^{-1/\gamma}
$$

We also define the Legendre-Fenchel dual
\begin{equation*}
\tilde{U}(y)=\max_{x>0}[U(x)-xy]=U(I_{2}(y))-yI_{2}(y)
\end{equation*}
This function satisfies:
\begin{equation*}
\tilde{U}^{\prime }(y)=-I_{2}(y)~~~y\leq0
\end{equation*}
\paragraph{CK Section 6. The constrained and unconstrained optimization problems.}
We define:%
\begin{equation*}
\mathcal{A}^{\prime }(i)=\{\alpha \in \mathcal{A}_{0}(i)|0\leq \alpha \leq
1\}
\end{equation*}
The supremum of the unconstrained problem is denoted by $V_{0}$, while the
supremum of the constrained problem is denoted by $V$, namely:
\begin{eqnarray*}
V_{0}(i) &=&\sup_{\alpha \in \mathcal{A}_{0}(i)}\mathbb E[U(I^{i,\alpha
}(T))|I(0)=i] \\
V(i) &=&\sup_{\alpha \in \mathcal{A}^{\prime }(i)}\mathbb E[U(I^{i,\alpha
}(T))|I(0)=i]
\end{eqnarray*}
\paragraph{CK Section 7. Solution of the unconstrained problem.}
We note that the expectation
\begin{equation*}
\mathcal{X}_{0}(y)\equiv \mathbb E[H^{(0)}(T)I_{2}(yH^{(0)}(T))]
\end{equation*}
is finite for every  $y\in (-\infty ,0]$. We define its inverse $\mathcal{Y}_{0}$:%
\begin{equation*}
\mathcal{Y}_{0}(\mathcal{X}_{0}(y))=y
\end{equation*}
The solution of the unconstrained problem is well-known, and equal to:
$$
\alpha(s) =\frac{\theta}{\sigma\gamma}=\frac{r-b}{\sigma^2|\gamma|}
$$
\paragraph{CK Section 8. Auxiliary unconstrained optimization problems.}
Recall $\delta(\nu)$ in (\ref{delta_function}). It is easily seen that:
$$
\alpha\nu-\delta(\nu)=\left\{\begin{array}{ll}
    \alpha\nu & \nu<0 \\
   (\alpha-1)\nu  & \nu>0
\end{array}\right\}\leq0
$$
We introduce a new process $I^{(\nu )}$ by:
\begin{equation*}
\frac{dI^{(\nu )}(t)}{I^{(\nu )}(t)}=(r+\alpha (t)\nu (t)-\delta (\nu(t)))dt+\alpha (t)\sigma dB_{2}^{(0)}(t)
\end{equation*}
Likewise we introduce
\begin{eqnarray*}
&&\theta ^{(\nu )} =\theta+\nu/\sigma \\
&&B^{(\nu )}(t) =B_2(t)+\int_{0}^{t}\theta ^{(\nu )}(s)ds \\
&&H^{(\nu )}(t) =\exp \Big(-rt+\int_{0}^{t}\delta (\nu (s))ds\Big)\mathcal{E}\Big(%
-\int_{0}^{t}\theta ^{(\nu )}(s)dB_2(s)\Big)\\
&&\mathcal{E}\Big(%
-\int_{0}^{t}\theta ^{(\nu )}(s)dB_2(s)\Big)\equiv \exp
\Big(-\int_{0}^{t}\theta ^{(\nu )}(s)dB_2(s)-\frac{1}{2}\int_{0}^{t}(\theta^{(\nu )}(s))^{2}ds\Big)
\end{eqnarray*}
We denote by $\mathcal{A}_{\nu }^{\prime }(i)$ the class of $\alpha $ for which%
\begin{equation*}
I_{\nu }^{i,\alpha }(t)\leq 1
\end{equation*}

Since the solution of our dual problem will have $\alpha(t)\nu (t)-\delta
(\nu (t))\leq 0$, clearly $\mathcal{A}^{\prime }(i)\subset \mathcal{A}_{\nu
}^{\prime }(i)$. We define:%
\begin{equation*}
V_{\nu }(i)=\sup_{\pi \in \mathcal{A}_{\nu }^{\prime }(i)}\mathbb E[U(I^{i,\alpha
}(T))]
\end{equation*}
\begin{equation*}
\mathcal{X}_{\nu }(y)\equiv \mathbb E[H^{(\nu )}(T)I_{2}(yH^{(\nu )}(T))]
\end{equation*}

We define a class of progressively measurable processes $\nu $ in $\mathbb{R}
$ by:
\begin{eqnarray*}
\mathcal{D}' =\Big\{\nu ;\mathbb E\int_{0}^{T}\delta (\nu(t))dt\leq \infty
,\mathbb E\int \nu ^{2}(t)dt<\infty ,\mathcal{X}_{\nu }(y)<\infty ,y\in
(-\infty,0] \Big\}
\end{eqnarray*}

Proposition 8.3. in CK shows that, if for some $\lambda \in $ $\mathcal{D}^{\prime
}$ the corresponding control $\alpha _{\lambda }$ is optimal for the
auxiliary optimization problem and if 
\begin{equation*}
-\delta (\lambda )+\alpha _{\lambda }(t)\lambda (t)=0
\end{equation*}
then $\alpha \in \mathcal{A}^{\prime }(i)$ and is optimal for the
constrained problem.

The solution of the unconstrained problem is:
\begin{equation}
    \alpha(s) =\frac{\theta^{(\nu)}}{\sigma\gamma}=\frac{\theta+\nu/\sigma}{\sigma\gamma}=\frac{r-b-\nu}{\sigma^2|\gamma|}
    \label{alpha_solution}
\end{equation}

\paragraph{CK Section 9. Contingent claims attainable by constrained portfolios.}
We sketch the proof of theorem 9.1 in CK, as the signs are different, and
the structure of the control is slightly different.
\begin{namedtheorem}[CK 9.1]
Let $B$ be a positive $\mathcal{F}_{T}$-measurable random
variable and suppose there is a process $\lambda \in \mathcal{D}'$ such that,
for all $\nu \in \mathcal D'$%
\begin{equation}
\mathbb E[H^{(\nu )}(T)B]\leq \mathbb E[H^{(\lambda )}(T)B]:=i  \label{maincond}
\end{equation}
Then there exists a control $\alpha \in \mathcal{A}^{\prime }(i)$ such that $%
I^{i,\alpha }=B$.
\end{namedtheorem}
\begin{proof}[Sketch of Proof]
See CK p.782 for a definition of the stopping time $\tau_n$.
By (\ref{maincond}) and subadditivity of $\delta $ (\ref{superadditive}):
\begin{eqnarray*}
0 &\leq &\lim_{\varepsilon \downarrow 0}\sup \frac{1}{\varepsilon }%
\mathbb E[(H^{(\lambda )(T)}-H^{(\lambda+\varepsilon (\nu -\lambda ) )}(T))B]  \notag \\
&=&\lim_{\varepsilon \downarrow 0}\sup \frac{1}{\varepsilon }\mathbb E\Big[H^{(\lambda)}(T)B \Big(1-\exp\Big(\notag\\
&&~~~\int_{0}^{T\wedge\tau_n}\left( \delta (\lambda (t)+\varepsilon (\nu
(t)-\lambda (t)))-\delta (\lambda (t))\right) dt\notag \\
&&~~~\times\mathcal{E}\big(%
\int_{0}^{T\wedge\tau_n}( -\theta ^{(\lambda )}(t)+\theta ^{(\lambda +\varepsilon
(\nu -\lambda ))}(t)) dB_{2}^{(\lambda )}(t)\big)\Big)\Big)\Big]  \notag \\
&\leq &\lim_{\varepsilon \downarrow 0}\sup \mathbb E\Big[H^{(\lambda )}(T)B( L_T+N_T) \Big]  
\end{eqnarray*}%
where
\begin{eqnarray*}
&&\breve{\delta}^{(\nu)}(\lambda(t))=\left\{\begin{array}{ll}
    \delta(\lambda(t)) &  \nu=0\\
    -\delta(\nu(t)-\lambda(t)) & \text{otherwise}
\end{array}\right.\\
&&L_T = \int_{0}^{T\wedge\tau_n}\breve{\delta}^{(\nu)}(\lambda(t))dt\\
&& N_T = \int_{0}^{T\wedge\tau_n}\frac{\nu (t)-\lambda (t)}{\sigma }dB_{2}^{(\lambda
)}(t)
\end{eqnarray*}

By Ito's lemma.

\begin{eqnarray*}
&&d[ H^{(\lambda )}(t)I(t)( L_t+N_t)] = I(t)H^{(\lambda )}(t)d( L_t+N_t)\\
&&~~~+( L_t+N_t)H^{(\lambda )}(t) I(t)\alpha(t)\sigma dB_{2}^{(\lambda
)}(t)+I(t)H^{(\lambda )}(t)\alpha(t)(\nu(t)-\lambda(t))dt
\end{eqnarray*}
which implies
\begin{eqnarray*}
&&H^{(\lambda )}(T)I(T)( L_T+N_T) \\
&&= \int_{0}^{\tau_n}I(t)H^{(\lambda )}(t)\Big(\frac{\nu(t)-\lambda(t)}{\sigma}+( L_t+N_t)\sigma\alpha(t)\Big)dB_{2}^{(\lambda
)}(t)\\
&&~~~+\int_{0}^{\tau_n}H^{(\lambda )}(t) I(t)\Big(\alpha(t)(\nu(t)-\lambda(t))dt+dL_t\Big)
\end{eqnarray*}
Therefore,
\begin{eqnarray*}
0\leq\mathbb E[H^{(\lambda )}(T) B(L_T+N_T)]=\mathbb E\Big[\int_{0}^{\tau_n}H^{(\lambda )}(t) I(t)\Big(\alpha(t)(\nu(t)-\lambda(t))dt+dL_t\Big)\Big]
\end{eqnarray*}

It is easy to see that, for any $\rho \in \mathcal{D}'$, take $\nu = \lambda+\rho$:
\begin{equation}
-\delta (\rho (t))+\alpha (t)\rho (t)\geq 0  \label{bizarre}
\end{equation}
and, taking $\nu (t)=0$, we  obtain:%
\begin{equation*}
-\delta (\lambda (t))+\alpha (t)\lambda (t)\leq 0
\end{equation*}
which together with (\ref{bizarre})\ for $\rho =\lambda $ yields:
\begin{equation*}
-\delta (\lambda (t))+\alpha (t)\lambda (t)=0
\end{equation*}
\end{proof}

\paragraph{CK Section 10. Equivalent optimality conditions.}
The most important implication to prove is (D)$\Rightarrow $(B)$\Rightarrow $%
(A)\ in CK. It shows that the solution of the dual problem solves the
auxiliary problem, and that, moreover, it is feasible and optimal for the
original constrained problem. We make it more explicit here.
\begin{namedtheorem}[(Part of) CK 10.1]
\label{CK_thm_101}
Suppose that for every $\nu \in 
\mathcal{D}'$,
\begin{equation}
\mathbb E[\tilde{U}(\mathcal{Y}_{\lambda }(i)H^{(\lambda )}(T))]\leq \mathbb E[\tilde{U}(%
\mathcal{Y}_{\lambda }(i)H^{(\nu )}(T))]  \label{CondTh101}
\end{equation}
then there exists a control $\alpha _{\lambda }\in \lbrack 0,1]$ that is optimal for the constrained problem $V_{\lambda
}(i)=\mathbb E[U(I^{i,\alpha _{\lambda }}(T))]$ and such that
\begin{equation*}
V_{\lambda }(i)=V(i)
\end{equation*}
\end{namedtheorem}
\begin{proof}
\begin{equation*}
\mathbb E[\tilde{U}(\mathcal{Y}_{\lambda }(i)H^{(\lambda )}(T))]\leq \mathbb E[\tilde{U}(%
\mathcal{Y}_{\lambda }(i)H^{(\lambda +\varepsilon (\nu -\lambda ))}(T))]
\end{equation*}
Since $\tilde{U}^{\prime }(y)=-I_{2}(y)$,
\begin{eqnarray*}
0 &\leq &\lim_{\varepsilon \downarrow 0}\sup \frac{1}{\varepsilon }\mathbb E[\tilde{U%
}(\mathcal{Y}_{\lambda }(i)H^{(\lambda+\varepsilon (\nu -\lambda ) )}(T))-\tilde{U}(\mathcal{Y}_{\lambda}(i)H^{(\lambda )}(T))] \\
&=&\mathcal{Y}_{\lambda
}(i)\lim_{\varepsilon \downarrow 0}\sup \frac{1}{\varepsilon }\mathbb E[I_{2}(%
\mathcal{Y}_{\lambda }(i)H^{(\lambda )}(T))(H^{(\lambda )}(T)-H^{(\lambda +\varepsilon (\nu
-\lambda ))}(T))
\end{eqnarray*}
By theorem 9.1 there exists a control $\alpha _{\lambda }\in \mathcal{A}%
_{\lambda }^{\prime }(i)$ such that:%
\begin{equation*}
I^{i,\alpha _{\lambda }}(T)=I_{2}(\mathcal{Y}_{\lambda }(i)H_{\lambda }(T))
\end{equation*}%
Clearly $\alpha _{\lambda }$ is optimal for the constrained problem, and 
\begin{equation*}
-\delta (\lambda )+\alpha _{\lambda }(t)\lambda (t)=0
\end{equation*}
Thus by proposition 8.3, $\alpha _{\lambda }$ is optimal for the constrained
problem.
\end{proof}
\paragraph{CK Section 12. A dual problem.}

Define:%
\begin{equation*}
\hat{V}(y)=\inf_{\nu \in \mathcal{D}'}\mathbb E[\tilde{U}(yH_{\nu }(T))]
\end{equation*}
In our case,%
\begin{equation*}
\tilde{U}(y)=\max_{x>0}\big[-\frac{x^{1-\gamma }}{1-\gamma }-xy\big]
\end{equation*}
Thus%
\begin{equation*}
y=U^{\prime }(x)=-x^{-\gamma }\implies
I_{2}(y)=(-y)^{-1/\gamma }
\end{equation*}
Let $\rho =(1-\gamma )/\gamma $. Then:
\begin{equation*}
\tilde{U}(y)=-(-y)^{-\rho}/\rho
\end{equation*}%
Typically, $\gamma =-1$, so that:
\begin{equation*}
\tilde{U}(y)=y^{2}/2
\end{equation*}

The main problem in condition (\ref{CondTh101}) is to find the optimal
process $H^{(\lambda )}$ (across all $H^{(\nu )}$) but it depends on $y$
which depends on $\lambda $. Thus the dual must be fixed for a fixed but
arbitrary real number $y$. The objective has the form
\begin{equation*}
\mathbb E[\tilde{U}(yH^{(\nu )}(T))]=\mathbb E[U(I_2(yH^{(\nu )}(T)))-yH^{(\nu
)}(T)I_{2}(yH^{(\nu )}(T))]
\end{equation*}

The right handside of the equation (see \citet[ p.134]{korn1997optimal}) is the maximum of the
function $h(B,y):=L(B,y)$ for all non-negative $\mathcal{F}_{T}$ measurable $%
B$ with $\mathbb E[H^{(\nu )}(T)B]\leq i$. Thus a minimization over all positive
numbers $y$ of $h(B,y)$ would yield the optimal utility of the unconstrained
problem. We could thus first minimize $\mathbb E[\tilde{U}(yH^{(\nu )}(T))]$ in $y$,
and then minimize over $\nu $. However, the main idea is to first minimize
over $\mu $, and then minimize over $y$, hoping that the two can be
interchanged.
\begin{namedproposition}[CK 12.1]
Suppose that for any $y$ there exists $\lambda _{y}$ such
that $\hat{V}(y)=\mathbb E[\tilde{U}(yH^{_{(\lambda _{y})}}(T))]$. Then there exist
an $\alpha \in \mathcal{A}^{\prime }(i)$ with $i=\mathcal{X}_{\lambda
_{y}(y)}$ which is optimal for the primal problem, and we have:
\begin{equation*}
\hat{V}(y)=\sup_{i}[V(i)-iy]
\end{equation*}
\end{namedproposition}
\begin{proof}
Write $\lambda $ for $\lambda _{\mathcal{Y}_{\lambda }(i)}$. Then
\begin{equation*}
\mathbb E[\tilde{U}(\mathcal{Y}_{\lambda }(i)H^{(\lambda )}(T))]\leq \mathbb E[\tilde{U}(%
\mathcal{Y}_{\lambda }(i)H^{(\nu )}(T))]
\end{equation*}
and we conclude by CK Theorem 10.1.
\end{proof}

\paragraph{CK Section 15. Deterministic coefficients and feedback formulae.}

Define:
\begin{equation*}
Q(y,t)=\mathbb E[\tilde{U}(yH^{(\nu )}(T))|yH^{(\nu )}(t)=y]
\end{equation*}
Recall
\begin{eqnarray*}
\frac{dH^{(\nu )}}{H^{(\nu )}} =(-r+\delta (\nu ))dt-(\theta+\nu/\sigma )dB_{2}
\end{eqnarray*}
The HJB equation is:%
\begin{eqnarray*}
&&\min_{\nu }\frac{1}{2}y^{2}(\theta+\nu/\sigma )^{2}Q_{yy}+  
y(-r+\delta (\nu))Q_{y}+Q_{t}=0   \\
&&Q(T,y)=\tilde{U}(y)  =-\frac{(-y)^{-\rho}}{\rho}
\end{eqnarray*}
Again, with $\rho=(1-\gamma)/\gamma<0$. We choose
\begin{equation*}
Q(y,t)=-\frac{1}{\rho }(-y)^{-\rho }v(t)  \label{ForYujia}
\end{equation*}
Thus
\begin{eqnarray*}
&&\frac{1}{2}y^{2}(\theta+\nu/\sigma )^{2}Q_{yy}+  
y(-r+\delta (\nu))Q_{y}\\
&&~~~=-\frac{1}{2}(\rho+1)(-y)^{-\rho}v(t)(\theta+\nu/\sigma)^2+(-r+\delta(\nu))(-y)^{-\rho}v(t)
\end{eqnarray*}
Dividing by $(-y)^{-\rho}v(t)$, the problem becomes:
\begin{equation}
\label{argmin}
\underset{\nu }{\operatorname{argmin}}-\frac{1+\rho }{2}(\theta+\nu/\sigma )^{2}+\delta (\nu)
\end{equation}
Recall that if $\nu$ is positive, then $\delta (\nu)=\nu$ thus we solve
(\ref{argmin}) and obtain
\begin{equation*}
 \nu=\frac{\sigma^2}{1+\rho }+r-b=-\sigma^2|\gamma|+r-b
\end{equation*}%
since $1+\rho=1/\gamma$ and $\gamma$ is negative. If $\nu$ is negative, then $\delta (\nu)=0$, thus $\nu = r-b$.

From (\ref{alpha_solution}), the solution is 
\begin{equation*}
    \alpha(s) =\frac{r-b-\nu}{\sigma^2|\gamma|}=\min \Big(1,\max \Big(0,\frac{r-b}{ \sigma^{2}|\gamma|}\Big)\Big)
\end{equation*}
Suppose $ \mu_0=\mu_1$ and treatment is better than no treatment $\bar{k}_1>K_0$. Thus
$r-b=\bar{k}_1-K_0$ is positive. Thus
\begin{equation*}
    \alpha(s) =\min \Big(1,\max \Big(0,\frac{\bar{k}_1-K_0}{ \sigma^{2}|\gamma|}\Big)\Big)
\end{equation*}
\section{Explicit Formula of \texorpdfstring{$A_{3}(\tau ,\gamma ) $}{TEXT} in (\ref{A3})}
\label{app::explicit_A3}
\begin{eqnarray*}
&&A_{3}(\tau ,\gamma ) =\int_{0}^{\tau }\left( \frac{\sigma _{x}^{2}}{2\gamma 
}+\lambda _{x}\bar{X}\right) A_{2}^{2}(s,\gamma )+\frac{\sigma _{x}^{2}}{2}%
A_{1}(s,\gamma )+(\gamma -1)\mu ds \\
&&= \left(\frac{\sigma _{x}^{2}}{2\gamma }+\lambda _{x}\bar{X}\right)\Bigg(\frac{2\lambda _{x}\bar{X}b_{2}(\gamma)A_{2}(\tau ,\gamma)}{\theta^3(\gamma)b_{3}(\gamma)}+\frac{2\bar{X}^2\lambda_x^2}{\theta^3(\gamma)}\Bigg(-\frac{A_{1}(\tau,\gamma )}{b_{3}(\gamma)}\notag\\
&&+\frac{8b_{1}^2(\gamma)\tau\log\left(\frac{2\theta(\gamma)-\left(\theta(\gamma)+b_{2}(\gamma)\right)\left(1-e^{-\theta(\gamma)\tau}\right)}{2\theta(\gamma)}\right) }{(b_{2}(\gamma)-\theta(\gamma))\sqrt{b_{1}(\gamma)b_{3}(\gamma)}}+\frac{b_{2}(\gamma)(\theta(\gamma)-2\sqrt{b_{1}(\gamma)b_{3}(\gamma)})}{\theta(\gamma)b_{3}^2(\gamma)} \notag\\ &&\times\log\left(\frac{b_{2}(\gamma)-2\sqrt{b_{1}(\gamma)b_{3}(\gamma)}}{\theta(\gamma)}\left(\frac{2b_{2}(\gamma)+4\sqrt{b_{1}(\gamma)b_{3}(\gamma)}e^{-\theta(\gamma)\tau/2}}{2\theta(\gamma)-\left(b_{2}(\gamma)+\theta(\gamma)\right)\left(1-e^{-\theta(\gamma)\tau}\right)}\right.\right.\notag\\
&&\left.\left.-\frac{(b_{2}(\gamma)+\theta(\gamma))A_{1}(\tau,\gamma)}{2b_{1}(\gamma)}\right)\right)+\frac{4b_{1}^2(\gamma)\tau}{(b_{2}(\gamma)+\theta(\gamma))^2}\Bigg)\Bigg) \notag\\
&& +\frac{\sigma _{x}^{2}}{2}\left(\frac{1}{b_{3}(\gamma)}\log\left(\frac{2\theta(\gamma) e^{-\theta(\gamma)\tau}}{2\theta(\gamma)-\left(b_{2}(\gamma)+\theta(\gamma)\right)\left(1-e^{-\theta(\gamma)\tau}\right)}\right)-\frac{2b_{1}(\gamma)\tau}{b_{2}(\gamma)+\theta(\gamma)}\right)\notag\\
&&+(\gamma-1)\mu\tau
\end{eqnarray*}
\section{Proof of Theorem \ref{thm::ou_moderate_prob}}
\label{app::proof_thm3}
We following the proof of Proposition 2 in \cite{gatto2021optimal}. Recall the equations (58) (59) in \cite{gatto2021optimal} and following same notations:
\begin{eqnarray}
&&\Big( \frac{\partial }{\partial t}+L_{1}\Big) f_{1} =0  \label{first_eq}
\\
&&\Big(\frac{\partial }{\partial t}+L_{1}\Big)f_{2} =-L_{2}f_{1}  \label{second_eq}
\end{eqnarray}
where $L_1, L_2$ are equations (53) (54) in \cite{gatto2021optimal}.
\paragraph{Solution of (\ref{first_eq})} 
We postulate that:
\begin{equation*}
f_{1}(Z,X,t)=Z^{1/\gamma }H_{1}(X,T-t) 
\end{equation*}
Substitution in (\ref{first_eq}) shows that $H_{1}$ solves:%
\begin{eqnarray}
&&\Big( \frac{\partial }{\partial t}+L^{\gamma }\Big) H_{1} =0
\label{dfdtplusL} \\
&&H_{1}(X,0) =1  \notag
\end{eqnarray}%
where the operator $L^{\gamma }$ is defined by: 
\begin{eqnarray*}
L^{\gamma }H &\equiv &\frac{1}{2}\sigma _{x}^{2}\frac{\partial ^{2}H}{%
\partial X^{2}}+\Big( \Big(\frac{\gamma -1}{\gamma }\sigma _{x}-\lambda
_{x}\Big)X+\lambda _{x}\bar{X}\Big) \frac{\partial H}{\partial X} \\
&&+\Big( X^{2}\Big( \frac{1}{2}\frac{1}{\gamma }\big(\frac{1}{\gamma }%
-1\big)\Big) +\mu \big(1-\frac{1}{\gamma }\big)\Big) H
\end{eqnarray*}
Using the Ansatz (\ref{KernelH1}), we can rewrite the LHS\ of (\ref{dfdtplusL}%
) into:
\begin{equation*}
(C_{1}(t)X^{2}+C_{2}(t)X+C_{3}(t))H_1/\gamma =0
\end{equation*}
Clearly all terms $C_{1},C_{2},C_{3}$ must be identically zero. Thus:
\begin{eqnarray*}
&&\frac{dA_{1}(t,\gamma)}{dt}=\frac{\sigma _{x}^{2}}{\gamma }A_{1}^{2}(t,\gamma)+2\big(\frac{%
\gamma -1}{\gamma }\sigma _{x}-\lambda _{x}\big)A_{1}(t,\gamma)+\frac{1-\gamma }{\gamma }\\
&&\frac{dA_{2}(t,\gamma)}{dt}=\frac{\sigma _{x}^{2}A_{1}(t,\gamma)}{\gamma }A_{2}(t,\gamma)+\big( 
\frac{\gamma -1}{\gamma }\sigma _{x}-\lambda _{x}\big) A_{2}(t,\gamma)+\lambda _{x}%
\bar{X}A_{1}(t,\gamma)\\
&&\frac{dA_{3}(t,\gamma)}{dt}=\frac{\sigma _{x}^{2}}{2}\big( A_{1}(t,\gamma)+\frac{A_{2}^{2}(t,\gamma)%
}{\gamma }\big) +\lambda _{x}\bar{X}A_{2}(t,\gamma)-\mu (1-\gamma )
\end{eqnarray*}
which admit the solutions (\ref{A1}),(\ref{A2}),(\ref{A3}).

\paragraph{Solution of (\ref{second_eq})}
The second equation can be rewritten
\begin{eqnarray}
\Big(\frac{\partial }{\partial t}+L_{1}\Big)f_{2} =\frac{1}{2}\frac{\beta ^{2}}{\gamma \sigma _{S}^{2}}Z^{2/\gamma
}SH_{1}(X,T-t)^{2} \label{diffequ_secondmodel}
\end{eqnarray}
We try the Ansatz:%
\begin{equation}
f_{2}(Z(t),X(t),t)=Z(t)^{2/\gamma }S(t)g(X(t),t)  \label{f2}
\end{equation}
Thus
\begin{eqnarray*}
&&\Big( \frac{\partial }{\partial t}+L^{\gamma /2}\Big) g(X,t) =\frac{1}{%
2}\frac{\beta ^{2}}{\sigma _{S}^{2}\gamma}H_{1}(X,T-t)^{2} \\
&&g(X,T) =0
\end{eqnarray*}
We use Lemma to obtain the $g(X,t)$ in (\ref{func_g}).

The optimal policy is:
\begin{eqnarray*}
\alpha ^{\ast }=\frac{1}{\sigma F}\left( \frac{\partial F}{\partial Z}XZ-%
\frac{\partial F}{\partial X}\sigma _{x}\right)  
=\alpha_0+\varepsilon\alpha_1+\mathcal O(\varepsilon ^{2}) 
\end{eqnarray*}
where
\begin{eqnarray*}
\alpha_0 &=& \frac{\partial f_1}{\partial Z}\frac{XZ}{\sigma f_1}-\frac{\partial f_1}{\partial X}\frac{\sigma_x}{\sigma f_1}\\
&=&\frac{X(t)}{\gamma\sigma}-\frac{\sigma_x}{\gamma\sigma}\left(A_{1}(T-t,\gamma)X(t)+A_{2}(T-t,\gamma)\right)
\end{eqnarray*}
\begin{eqnarray*}
\alpha_1 &=& \frac{XZ}{\sigma f_1}\left(\frac{\partial f_2}{\partial Z}-\frac{\partial f_1}{\partial Z}\frac{f_2}{f_1}\right)-\frac{\sigma_x}{\sigma f_1}\left(\frac{\partial f_2}{\partial X}-\frac{\partial f_1}{\partial X}\frac{f_2}{f_1}\right)\\
&=&\frac{Z^{1/\gamma}(t)S(t)}{H_1(X, T-t)\sigma}\left(\frac{g(X(t),t)X(t)}{\gamma}-\sigma_x\frac{\partial g}{\partial X}\right.\\
&&\left.+\sigma_x\frac{g(X(t),t)}{\gamma}\left(A_{1}(T-t,\gamma)X(t)+A_{2}(T-t,\gamma)\right)\right)
\end{eqnarray*}
\begin{lemma}
Let $u(x,t)=\frac{1}{2}\frac{\beta ^{2}}{\sigma _{S}^{2}\gamma}%
H_{1}(x,T-t)^{2}$. The solution to
\begin{eqnarray}
&&\frac{\partial g(x,t)}{\partial t}+L^{\gamma /2}g(x,t) =u(x,t)
\label{Tosolve} \\
&&g(x,T) =0  \nonumber
\end{eqnarray}
is in (\ref{func_g}).
\end{lemma}
\begin{proof}[Sketch of Proof]
The solution $g(x,t)$ is the price of a variable-coupon bond in an affine
model. The building block is the solution of a zero-coupon bond in the
similar model.
Define $m(x)$ and $r(x)$ to be such that:
\begin{eqnarray*}
&&L^{\gamma /2}f(x,t)=\frac{1}{2}\sigma _{x}^{2}\frac{\partial ^{2}f(x,t)}{%
\partial x^{2}}+m(x)\frac{\partial f(x,t)}{\partial x}-r(x)f(x,t)\\
&&m(x) =\Big(\frac{\gamma /2-1}{\gamma /2}\sigma _{x}-\lambda _{x}\Big)x+\lambda _{x}%
\bar{X} \\
&&r(x) =-\Big(x^{2}\frac{1}{\gamma }\big(\frac{2}{\gamma }-1\big)+\mu \big(1-\frac{2}{\gamma 
}\big)\Big)
\end{eqnarray*}
Let $f(x,t)$ be the solution of: 
\begin{equation}
\frac{\partial f(x,t)}{\partial t}+\frac{1}{2}\sigma _{x}^{2}\frac{\partial
^{2}f(x,t)}{\partial x^{2}}+m(x)\frac{\partial f(x,t)}{\partial x}=r(x)f(x,t)
\label{Thisisit}
\end{equation}
Defining:
\begin{equation}
dX(t)=m(X)dt+\sigma _{x}dW(t)  \label{dX}
\end{equation}
we see that:
\begin{equation}
\frac{\partial f(x,t)}{\partial t}+\frac{1}{2}\sigma _{x}^{2}\frac{\partial
^{2}f(x,t)}{\partial x^{2}}+m(x)\frac{\partial f(x,t)}{\partial x}%
=\mathbb E[df(X,t)|X(t)=x]/dt  \label{dP}
\end{equation}
Thus (\ref{Thisisit}) can be rewritten:%
\[
\mathbb E[df(X(t),t)-r(X(t))f(X(t),t)dt|X(t)]=0
\]
Using the integrating factor $\exp (-\int_{0}^{t}r(X(s)ds)$, we have:
\[
\mathbb E[d(\exp (-\int_{0}^{t}r(X(s))ds)f(X(t),t))|X(t)]=0
\]
Under the boundary condition $f(X(T),T)=1$ the only possible solution is:
\[
f(x,t;T)=\mathbb E[\exp (-\int_{t}^{T}r(X(s))ds)|X(t)=x]
\]
Define $P(t,T)=f(X(t),t;T)=H_{2}(X(t),T-t)$  to be the
price of a discount bond with a maturity of $T$. Clearly:
\[
\frac{dP(t,T)}{P(t,T)}=r(X(t))dt+v(t,T)dW(t)
\]
where:%
\[
v(t,T)=\sigma _{x}\frac{\frac{\partial f}{\partial x}}{f}
\]
By Ito's lemma, and for the exact same reason as (\ref{dP}):
\[
\frac{\partial g(x,t)}{\partial t}+\frac{1}{2}\sigma _{x}^{2}\frac{\partial
^{2}g(x,t)}{\partial x^{2}}+m(x)\frac{\partial g(x,t)}{\partial x}%
=\mathbb E[dg(X,t)|X(t)=x]/dt
\]
The stochastic equivalent of (\ref{Tosolve}) is:
\[
\mathbb E[dg(X(t),t)-r(X(t))g(x,t)dt|X(t)]=\mathbb E[u(X(t),t)dt|X(t)]
\]
The solution is:%
\[
g(X(t),t)=\int_{\tau =t}^{T}Q(t,\tau )d\tau 
\]
where:%
\[
Q(t,\tau )=\mathbb E[\exp (-\int_{t}^{\tau }r(X(s))ds)u(X(\tau ),\tau )|X(t)]
\]
Clearly, for some volatility $\sigma _{Q}(t,\tau )$%
\[
\frac{dQ(t,\tau )}{Q(t,\tau )}=r(X(t))dt+\sigma _{Q}(t,\tau )dW(t)
\]%
We are now ready to define a change of numeraire. Let
\[
dW^{\tau }=dW-v(t,\tau )dt
\]
By Theorem 9.2.2. in \cite{shreve2004stochastic}, $Q(t,\tau )/P(t,\tau )$ is a $\mathbb{P}%
^{\tau }$-martingale, i.e.,%
\[
Q(t,\tau )=P(t,\tau )\mathbb E_{t}^{\tau }[u(X(\tau )]
\]
where
\begin{eqnarray*}
dX(t) &=&m(X)dt+\sigma _{x}dW(t) \\
&=&m(X)dt+\sigma _{x}(dW^{\tau }(t)+v(t,\tau )dt)
\end{eqnarray*}
From (\ref{KernelH1}),
\begin{eqnarray*}
u(X(t),t) =\frac{1}{2}\frac{\beta ^{2}}{\sigma _{{S}^{2}}\gamma}e^{\frac{2}{^{\gamma }}\left(%
\frac{A_{1}(T-t,\gamma)}{2}X(t)^{2}+A_{2}(T-t,\gamma)X(t)+A_{3}(T-t,\gamma)\right)} 
\end{eqnarray*}%
Let us now take: 
$$
P(t,T) =\exp \Big(\frac{2}{^{\gamma }}\Big(\frac{A_{1}(T-t,\gamma/2)}{2}%
X^{2}(t)+A_{2}(T-t,\gamma/2)X(t)+A_{3}(T-t,\gamma/2)\Big)\Big)
$$
Thus:%
\begin{equation*}
    v(t,\tau )=\frac{\sigma _{x}}{^{\gamma }}(A_{1}(\tau -t,\gamma/2)X(t)+A_{2}(\tau
-t,\gamma/2))
\end{equation*}
\begin{gather}
\label{Pof}
\begin{aligned}
    dX(t)
    =&\Big[ \big(\frac{\frac{\gamma}{2} -1}{\gamma /2}\sigma _{x}-\lambda _{x}+\frac{%
    \sigma _{x}^{2}}{^{\gamma }}A_{1}(\tau -t,\gamma/2)\big)X(t)+\lambda _{x}\bar{X}\\
    &+\frac{%
    \sigma _{x}^{2}}{^{\gamma }}A_{2}(\tau -t,\gamma/2)\Big] dt +\sigma _{x}dW^{\tau }(t)
\end{aligned}
\end{gather}

Thus $\mathbb E_{t}^{\tau }[u(X(\tau ))]$ when (\ref{Pof}) holds can be calculated
exactly the same way as $\mathbb E[u(X(\tau ))]$ when (\ref{dX}) holds. The
structure is also affine, and there will be a solution of the form:
\[
\mathbb E_{t}^{\tau }[u(X(\tau ),\tau )]=\frac{1}{2}\frac{\beta ^{2}}{\sigma
_{{S}^{2}}\gamma}\mathbb E_{t}^{\tau }\left[e^{\frac{2}{\gamma} \left(\frac{A_{1}(T-\tau,\gamma )}{2}X^{2}(\tau )+A_{2}(T-\tau,\gamma
)X(\tau )+A_{3}(T-\tau,\gamma )\right)}\right]
\]
To summarize, since $%
P(t,T)=H_{2}(X(t),T-t)$
\begin{gather}
\label{eqn::to_get_g}
\begin{aligned}
g(x,t) =&\int_{\tau =t}^{T}P(t,\tau )\mathbb E_{t}^{\tau }[u(X(\tau ),\tau )]d\tau 
\\
=&\int_{\tau =t}^{T}H_{2}(X,\tau -t)\frac{1}{2}\frac{\beta ^{2}}{\sigma
_{{S}^{2}}\gamma}
\\
&\mathbb E_{t}^{\tau }\left[e^{ \frac{2}{\gamma}\left(\frac{A_{1}(T-\tau,\gamma
)}{2}X^{2}(\tau )+A_{2}(T-\tau ,\gamma)X(\tau )+A_{3}(T-\tau,\gamma )\right)}\right]d\tau
\end{aligned}
\end{gather}

Let $\tilde{M}(t,\tau )$ as in (\ref{M_tilde}) and
\[
Y(\tau )=X(\tau )+\frac{A_{2}(T-\tau,\gamma )}{
A_{1}(T-\tau,\gamma )}
\]
Clearly:
\begin{eqnarray*}
&&\mathbb E^{\tau }[X(\tau )|X(t) =x]=x\tilde{M}(t,\tau )+\int_{s=t}^{\tau }\tilde{M}(s,\tau )(\lambda _{x}\bar{X}+\frac{\sigma _{x}^{2}}{^{\gamma }}%
A_{2}(\tau -s,\gamma))ds \\
&&\mathbb Var^{\tau }[X(\tau )|X(t) =x]=\sigma _{x}^{2}\int_{t}^{\tau }\tilde{M}^2(s,\tau )ds
\end{eqnarray*}
Thus we can calculate:%
\begin{eqnarray*}
&&m_{Y}(\tau ,x) =\mathbb E^{\tau }[Y(\tau )|X(t)=x]\notag\\
&&\hspace{1.5cm}=\mathbb E^{\tau
}[X(\tau )|X(t)=x]+\frac{A_{2}(T-\tau,\gamma )}{A_{1}(T-\tau,\gamma )} \label{m-in-g}\\
&&V_{Y}(\tau ,x) =\mathbb Var^{\tau }[Y(\tau )|X(t)=x]=\sigma
_{x}^{2}\int_{t}^{\tau }\tilde{M}^2(s,\tau )ds \label{v-in-g}
\end{eqnarray*}
We can further develop:
\begin{eqnarray*}
&&\mathbb E_{t}^{\tau }\Big[\exp \Big(\frac{2}{\gamma}\Big(\frac{A_{1}(T-\tau ,\gamma)}{2}X^{2}(\tau )+A_{2}(T-\tau ,\gamma)X(\tau
)+A_{3}(T-\tau,\gamma )\Big)\Big)\Big] \\
&&=\mathbb E_{t}^{\tau }\left[e^{\frac{2}{\gamma}A_{3}(T-\tau,\gamma  )+\frac{1}{\gamma}A_{1}(T-\tau,\gamma  )\left(X(\tau)+\frac{A_{2}(T-\tau,\gamma  )}{A_{1}(T-\tau,\gamma  )}\right)^2-\frac{A_{2}^2(T-\tau ,\gamma )}{\gamma A_{1}(T-\tau,\gamma )}}\right]\\
&&=\mathbb E_{t}^{\tau }\left[e^{\frac{2}{\gamma}A_{3}(T-\tau,\gamma  )-\frac{A_{2}^2(T-\tau ,\gamma )}{\gamma A_{1}(T-\tau ,\gamma )}}e^{\frac{A_{1}(T-\tau,\gamma  )}{\gamma}\left(X(\tau)+\frac{A_{2}(T-\tau ,\gamma )}{A_{1}(T-\tau,\gamma  )}\right)^2}\right]\\
&&=e^{\frac{2}{\gamma}A_{3}(T-\tau,\gamma  )-\frac{A_{2}^2(T-\tau ,\gamma )}{\gamma A_{1}(T-\tau,\gamma  )%
}}\frac{1}{\sqrt{2\pi V_{Y}(\tau ,x)}}\int e^{\frac{A_{1}(T-\tau,\gamma  )}{\gamma}y^2}e^{-%
\frac{(y-m_{Y}(\tau ,x))^{2}}{2V_{Y}(\tau ,x)}}dy\\
&&=e^{\frac{2}{\gamma}A_{3}(T-\tau,\gamma  )-\frac{A_{2}^2(T-\tau,\gamma)}{\gamma A_{1}(T-\tau,\gamma)%
}+\frac{m^2_{Y}(\tau ,x)A_{1}(T-\tau,\gamma )}{\gamma-2V_{Y}(\tau ,x)A_{1}(T-\tau,\gamma )}}\frac{1}{\sqrt{1-2 V_{Y}(\tau ,x)A_{1}(T-\tau,\gamma)/\gamma}}
\end{eqnarray*}
providing $\gamma<2A_{1}(T-\tau ,\gamma )V_{Y}(\tau ,x)$. Thus equation (\ref{func_g}) follows from equation (\ref{eqn::to_get_g}).
\end{proof}

\section{Proof of Theorem \ref{thm::constant_moderate}}
\label{app::proof_thm4}
When $X(t)$ is a constant, equations (53) and (54) in \cite{gatto2021optimal} become
\begin{eqnarray*}
&&L_{1}F =
\frac{1}{2}Z^{2}\big(\frac{b-r}{\sigma}\big)^{2}\frac{\partial ^{2}F}{\partial Z^{2}}-\mu Z\frac{\partial F}{\partial Z}+\mu F  \notag \\
&&L_{2}F =-\frac{1}{2}\frac{\beta ^{2}}{\sigma _{S}^{2}}ZSF\frac{\partial F}{%
\partial Z} 
\end{eqnarray*}
Use the Ansatz $f_1(Z(t), t) = Z^{1/\gamma}(t)h_1(T-t)$ and insert in (\ref{first_eq_full_expansion}) shows that $h_1$ solves:
\begin{eqnarray}
\left( \frac{\partial }{\partial t}+L^{\gamma }\right) h_{1} =0~~~~~~~h_{1}(0) =1
\label{sim-dfdtplusL}
\end{eqnarray}%
where the operator $L^{\gamma }$ is defined by: 
\begin{eqnarray*}
L^{\gamma }H &\equiv &\Big( \big(\frac{b-r}{\sigma}\big)^{2}\Big( \frac{1}{2}\frac{1}{\gamma }\big(\frac{1}{\gamma }%
-1\big)\Big) +\mu \big(1-\frac{1}{\gamma }\big)\Big) H
\end{eqnarray*}
Using the Ansatz (\ref{Kernelh}), we can rewrite  (\ref{sim-dfdtplusL}%
) into:
\begin{equation*}
\Big(C_{1}(t)\big(\frac{b-r}{\sigma}\big)^{2}+C_{2}(t)\Big)h_1/\gamma =0
\end{equation*}
Clearly all terms $C_{1},C_{2}$ must be identically zero. Thus
\begin{equation*}
\frac{da_{1,1}t}{dt}=\frac{1-\gamma }{\gamma }~~~~~\frac{da_{1,2}t}{dt}=\mu (\gamma -1)
\end{equation*}
which admit the solutions (\ref{ai1}), (\ref{ai2}) at $i=1$.

Now use $f_i(Z(t), t) = Z^{2^{i-1}/\gamma}(t)S^{2^{i-1}-1}(t)g_i(t)$. We can rewrite (\ref{second_eq_full_expansion}) by 
\begin{eqnarray}
\left( \frac{\partial }{\partial t}+L^{\gamma /2^i}\right) g_{i+1}(t) =\frac{1}{%
2}\frac{\beta ^{2}2^{i-1}}{\sigma _{S}^{2}\gamma}g_{i}^{2}(t) ~~~~~~~g_{i+1}(T) =0\label{sim-g2-equ}
\end{eqnarray}
Let $u(t)=\frac{1}{%
2}\frac{\beta ^{2}2^{i-1}}{\sigma _{S}^{2}\gamma}g_{i}^{2}(t)$ and
\begin{eqnarray*}
r_i = \frac{2^i}{\gamma}\Big(\frac{1}{2}\big(\frac{b-r}{\sigma}\big)^{2}a_{i+1,1}+a_{i+1,2}\Big)
\end{eqnarray*}
Then 
\[
L^{\gamma /2^i}g_{i+1}(t)=r_ig_{i+1}(t)
\]
and the stochastic equivalent of (\ref{sim-g2-equ})\ is:
\begin{eqnarray*}
\frac{\partial g_{i+1}(t)}{\partial t}+r_ig_{i+1}(t)=u(t)~~~~~~~g_{i+1}(T) =0
\end{eqnarray*}
which admits
\begin{equation}
    g_{i+1}(t) = -\frac{1}{2}\frac{\beta ^{2}2^{i-1}}{\sigma _{S}^{2}\gamma}\frac{1}{h_{i+1}(t)}\int_{t}^Tg_i^2(s)h_{i+1}(s)ds\label{g_i}
\end{equation}
We have showed that $g_1=h_1(T-t)$. Here we also provide the $g_2$ and $g_3$ in the following:
\begin{eqnarray*}
g_{2}(t) &=& \frac{\beta ^{2}}{2\sigma _{S}^{2}}\frac{h_2(T-t)-h_1^2(T-t)}{\big(\frac{b-r}{\sigma}\big)^{2}(a_{1,1}-a_{2,1})+2(a_{1,2}-a_{2,2})}
\end{eqnarray*}
\begin{eqnarray*}
&&g_{3}(t) = -\frac{\beta ^{6}}{16\sigma _{S}^{6}}\left(\frac{1}{\big(\frac{b-r}{\sigma}\big)^{2}(a_{1,1}-a_{2,1})+2(a_{1,2}-a_{2,2})}\right)^2\frac{1}{h_3(t)}\\
&&\left(h_2^2(T)\frac{\frac{h_3(T)}{h_2^2(T)}-\frac{h_3(t)}{h_2^2(t)}}{\frac{a_{3,1}-a_{2,1}}{2}\big(\frac{b-r}{\sigma}\big)^{2}+a_{3,2}-a_{2,2}}\right.\left.+h_1^4(T)\frac{\frac{h_3(T)}{h_1^4(T)}-\frac{h_3(t)}{h_1^4(t)}}{\frac{a_{3,1}-a_{1,1}}{2}\big(\frac{b-r}{\sigma}\big)^{2}+a_{3,2}-a_{1,2}}\right.\\
&&\left.-2\frac{h_2(T)}{h_1^2(T)}\frac{\frac{h_3(T)h_1^2(T)}{h_2(T)}-\frac{h_3(t)h_1^2(t)}{h_2(t)}}{\frac{a_{3,1}-\frac{a_{2,1}-a_{1,1}}{2}}{2}\big(\frac{b-r}{\sigma}\big)^{2}+a_{3,2}-\frac{a_{2,2}-a_{1,2}}{2}}\right)
\end{eqnarray*}

Suppose we use the first two expansions, the optimal policy is given by:
\begin{eqnarray*}
\alpha ^{\ast }=\alpha_0+\varepsilon\alpha_1+\mathcal O(\varepsilon ^{2}) 
\end{eqnarray*}
where 
\begin{eqnarray*}
&&\alpha_0 = \frac{\partial f_1}{\partial Z}\frac{\frac{b-r}{\sigma}Z}{\sigma f_1}=\frac{\frac{b-r}{\sigma}}{\gamma\sigma}\\
&&\alpha_1 =\frac{\frac{b-r}{\sigma}Z}{\sigma f_1}\left(\frac{\partial f_2}{\partial Z}-\frac{\partial f_1}{\partial Z}\frac{f_2}{f_1}\right)=\frac{Z^{1/\gamma}(t)S(t)}{h_1(T-t)\sigma}\frac{g_2(t)\frac{b-r}{\sigma}}{\gamma}
\end{eqnarray*}

\end{appendices}

\bibliographystyle{plainnat}
\bibliography{reference}

\end{document}